\documentclass[11pt,leqno]{amsart}
 \usepackage{graphicx}    
 \usepackage{amsmath}
\usepackage{amssymb}
\usepackage[T1]{fontenc}
\usepackage[latin1]{inputenc}
\usepackage[english]{babel}
\usepackage[arrow, matrix, curve]{xy}
\usepackage{amsthm}
\usepackage{amsmath,amscd}
\usepackage{epic,eepic}
\usepackage{bbm}
\usepackage{psfrag}

\numberwithin{equation}{section}

\DeclareMathOperator{\GL}{GL}
\DeclareMathOperator{\PGL}{PGL}

\renewcommand{\phi}{\varphi}

\newcommand{\Acal}{\mathcal{A}}
\newcommand{\Bcal}{\mathcal{B}}

\newcommand{\Fcal}{\mathcal{F}}
\newcommand{\Gcal}{\mathcal{G}}

\newcommand{\Ocal}{\mathcal{O}}

\newcommand{\Abb}{\mathbb{A}}
\newcommand{\Fbb}{\mathbb{F}}
\newcommand{\Gbb}{\mathbb{G}}
\newcommand{\Pbb}{\mathbb{P}}
\newcommand{\Q}{\mathbb{Q}}
\newcommand{\R}{\mathbb{R}}
\newcommand{\Z}{\mathbb{Z}}

\newcommand{\Mfrak}{\mathfrak{M}}
\newcommand{\Nfrak}{\mathfrak{N}}

\newcommand{\Spec}{{\rm Spec}}

\newcommand{\id}{{\rm id}}

\newcommand{\dete}{{\rm det}}
\newcommand{\Res}{{\rm Res}}

\newcommand{\pr}{{\rm pr}}

\newtheorem{theo}{Theorem}[section]
\newtheorem{lem}[theo]{Lemma}
\newtheorem{prop}[theo]{Proposition}
\newtheorem{cor}[theo]{Corollary}
\theoremstyle{remark}
\newtheorem{rem}[theo]{Remark}
\theoremstyle{remark}

\theoremstyle{definition}

\begin{document}

\title[Connectedness of Kisin varieties]{Connectedness of Kisin varieties for $\GL_2$}
\author[E. Hellmann]{Eugen Hellmann}
\begin{abstract}
We show that the Kisin varieties associated to simple $\phi$-modules of rank $2$ are connected in the case of an arbitrary cocharacter.
This proves that the connected components of the generic fiber of the flat deformation ring of an irreducible $2$-dimensional Galois representation of a local field 
are precisely the components where the multiplicities of the Hodge-Tate weights are fixed. 
\end{abstract}
\maketitle

\section{Introduction}

In his paper \cite{Kisin}, Kisin constructs a projective scheme over a finite field of characteristic $p$ whose closed points parametrize the extensions of 
a fixed Galois representation of a local $p$-adic field $K$ in a vector space over a finite field $\Fbb$ of characteristic $p$ to a finite flat group scheme over the ring of integers $\Ocal_K$ of $K$.
In \cite{phimod}, Pappas and Rapoport name these varieties \emph{Kisin varieties}.
Kisin shows that the connected components of this scheme are in bijection with the connected components of the generic fiber of the flat deformation 
ring of the fixed representation in the sense of Ramakrishna (cf. \cite{Rama}).  
The quotient of the flat deformation ring corresponding to those flat representations whose Hodge-Tate weights $0$ and $1$ have the same
multiplicity is of particular interest for modularity lifting theorems (see \cite{Kisin}).
In the case where $K$ is totally ramified over $\Q_p$ and the Galois representation is $2$-dimensional, the connected components of the corresponding variety over $\Fbb_p$ were determined 
by Kisin in \cite[2.5]{Kisin}. In general Kisin conjectures that the connected components are given by open and closed subschemes on which the rank of the maximal multiplicative subobject and the maximal \'etale quotient are fixed, see \cite[2.4.16]{Kisin}. Kisin's connectedness result was generalized by Gee and Imai to the case of an arbitrary local field $K$ (again in the $2$-dimensional case), see \cite{Gee} and \cite{Imai1}, but their results again only concern the case corresponding to deformations whose Hodge-Tate weights $0$ and $1$ (in the generic fiber) have the same multiplicity. 
In \cite{Imai2} and \cite{Imai3}, Imai also determines the type of the zeta function of the Kisin variety in these cases and counts the points
of the variety parametrising extensions of the trivial representation.
In the case of low ramification (more precisely if the ramification index of $K$ is smaller than $p-1$), the scheme is either empty or contains a single point,
by Raynaud's theorem \cite[2.2.3, 3.3.2]{Raynaud}. 
In this paper we prove connectedness of the Kisin variety corresponding to arbitrary Hodge-Tate weights in the case of an irreducible Galois representation of dimension $2$. This generalises our result in \cite{ich}, that the Kisin variety for arbitrary Hodge-Tate weights is geometrically connected if the representation is absolutely irreducible and $K$ is totally ramified.

We now describe our main result.
Let $K$ be a finite extension of $\Q_p$ with residue field $k$. Fix a uniformizer $\pi$ and a compatible system of $p^n$-th roots $\pi_n$ of $\pi$ in a fixed algebraic closure $\bar K$ of $K$. We denote by $K_\infty$ the subfield of $\bar K$ obtained from $K$ by adjoining $\pi_n$ for all $n$. Then the absolute Galois group $G_{K_\infty}$ of $K_\infty$ is isomorphic to the absolute Galois group of a local field in characteristic $p$ and hence the $\Fbb$-linear representations of $G_{K_\infty}$ are described in terms of \'etale $\phi$-modules over $(k\otimes_{\Fbb_p}\Fbb)((u))$ (or equivalently $\phi^n$-modules over $\Fbb((u))$ if $k\subset\Fbb$ and $n=[k:\Fbb_p]$), where $\Fbb$ is a finite field of characteristic $p$, compare \cite[A]{Fontaine}. Here the Frobenius $\phi$ acts on $(k\otimes_{\Fbb_p}\Fbb)((u))$ as the $p$-power map on $k((u))$ and as the identity on $\Fbb$.

Assume that $p\geq3$. By a result of Kisin (building on work of Breuil), the finite flat group schemes over $\Spec\,\Ocal_K$ that are $p$-torsion are described in terms of free $k[[u]]$-modules
$\Mfrak$ together with a $\phi$-linear injection $\Phi:\Mfrak\rightarrow\Mfrak$ such that $u^e\Mfrak\subset \Phi(\phi^\ast\Mfrak)$. Here $e=[K:W(k)[1/p]]$ is the ramification index of $K$ over $\Q_p$.
Under this equivalence of categories, the restriction to $G_{K_\infty}$ of the Galois representation in the generic fiber corresponds up to twist to the \'etale $\phi$-module obtained from $(\Mfrak,\Phi)$ by inverting $u$.
More precisely, let $\Gcal\rightarrow \Spec\,\Ocal_K$ be the finite flat group scheme defined by $(\Mfrak,\Phi)$. We write $V=\Gcal(\bar K)$ for the Galois representation on the generic fiber and assume that $V$ carries an action of $\Fbb$. Then $(\Mfrak[1/u],\Phi)$ is the \'etale $\phi$-module associated to $V(1)|_{G_{K_\infty}}$.

Conversely let us fix a free $(\Fbb\otimes_{\Fbb_p}k)((u))$ module $N$ of rank $d$ together with an isomorphism $\Phi:\phi^\ast N\rightarrow N$. 
Kisin shows that there is a closed subvariety $C_K(\Phi)$ of the affine Grassmannian of the algebraic group $G=\Res_{k/\Fbb_p}\GL_d$ such that the $\Fbb'$-valued points of $C_K(\Phi)$ are the $(\Fbb'\otimes_{\Fbb_p}k)[[u]]$-lattices $\Mfrak'\subset N\widehat{\otimes}_\Fbb\Fbb'$ satisfying
\[u^e\Mfrak'\subset \Phi(\phi^\ast\Mfrak')\subset\Mfrak'.\]  
As the restriction of flat $p$-torsion representations of $G_K={\rm Gal}(\bar K/K)$ to $G_{K_\infty}$ is fully faithful \cite[Theorem 3.4.3]{Breuil}, this scheme indeed parametrises finite flat group scheme models of a given $G_K$-representation.

If we fix a dominant cocharacter of the group $G$, then we can ask for the closed subscheme $C_\nu(\Phi)$ of the affine Grassmannian parametrising lattices $\Mfrak$ such that the elementary divisors of $\Phi(\phi^\ast\Mfrak)$ with respect to $\Mfrak$ are less or equal to $\nu$ in the usual order on dominant coweights (see also the definitions below). In this context $p=2$ is also allowed. We will prove the following Theorem.

\begin{theo}
Fix a dominant cocharacter $\nu$ and a simple $2$-dimensional \'etale $\phi$-module $(N,\Phi)$. Then the scheme $C_\nu(\Phi)$ is geometrically connected.
\end{theo}

We believe that $C_\nu(\Phi)$ is in fact irreducible and hope to come back to this question and also to discuss some of the structure of $C_\nu(\Phi)$.

Now let $p\geq 3$ and let $\bar\rho:G_K\rightarrow \GL_2(\Fbb)$ be an irreducible continuous representation of $G_K$ with coefficients in a finite extension $\Fbb$ of $\Fbb_p$.
Then the flat deformation functor is pro-representable by a complete local noetherian $W(\Fbb)$-algebra $R^{\rm fl}$.
Let 
\[\mu:\Gbb_{m,\bar\Q_p}\longrightarrow (\Res_{K/\Q_p}\GL_2)_{\bar\Q_p}\]
be a miniscule dominant cocharacter and write $R^{{\rm fl},\mu}$ for the quotient of $R^{\rm fl}[1/p]$ corresponding to those valuations whose Hodge-Tate weights are given by $\mu$. If the cocharacter is not miniscule, i.e. the Hodge-Tate weights are not in $\{0,1\}$, then the corresponding quotient of $R^{\rm fl}$ would be empty. Our main result has the following consequence.
\begin{cor}
The scheme $\Spec(R^{{\rm fl},\mu})$ is connected.
\end{cor}

{\bf Acknowledgements:} I thank M. Rapoport, X. Caruso and N. Imai for their remarks on a preliminary version of this paper and for their interest in this work.
The author was supported by the SFB/TR45 "Periods, Moduli spaces and Arithmetic of Algebraic Varieties" of the DFG (German Research Foundation).

\section{Notations}
Let $k$ be a finite field of characteristic $p>0$ of degree $n=[k:\Fbb_p]$ over $\Fbb_p$. We fix an algebraic closure $\bar\Fbb_p$ of $\Fbb_p$.
Let $G$ denote the reductive group $\Res_{k/\Fbb_p}\GL_2$ over $\Fbb_p$. Then
\begin{equation}\label{prodisom}
\begin{xy}\xymatrix{
G_{\bar\Fbb_p}=G\otimes_{\Fbb_p}\bar\Fbb_p\ar[r]^\cong &\prod_{i=1}^n(\GL_2)_{\bar\Fbb_p}.}
\end{xy}
\end{equation}
The automorphism of $\Res_{k/\Fbb_p}\GL_2$ induced by the absolute Frobenius on $k$ acts on the right hand side by shifting the factors.
The Weil restriction of the Borel subgroup of upper triangular matrices in $\GL_2$ defines a Borel subgroup $B\subset G$ such that
\[\begin{xy}\xymatrix{B\otimes_{\Fbb_p}\bar\Fbb_p\ar[r]^\cong & \prod_{i=1}^n B_i}\end{xy}\]
under the isomorphism in $(\ref{prodisom})$, where $B_i\subset (\GL_2)_{\bar\Fbb_p}$ is the subgroup of upper triangular matrices.\\
Denote by $LG$ (resp. $L^+G$) the loop group (resp. the positive loop group) of $G$, i.e. the ind-group scheme representing the functor $R\mapsto G(R((u)))$ (resp. $R\mapsto G(R[[u]])$) on the category of $\Fbb_p$-algebras. Further we denote by 
\[\Fcal_G=LG/L^+G\]
for the affine Grassmannian of the group $G$. The isomorphism in $(\ref{prodisom})$ induces isomorphisms
\[\begin{xy}\xymatrix{
\Fcal_G\otimes_{\Fbb_p}\bar\Fbb_p\ar[r]^{\hspace{-0.7cm}\cong} & L(G_{\bar\Fbb_p})/L^+(G_{\bar\Fbb_p})\ar[r]^{\hspace{5mm}\cong} & \prod_{i=1}^n\Fcal_i,
}\end{xy}\]
where $\Fcal_i\cong L(\GL_2)_{\bar\Fbb_p}/L^+(\GL_2)_{\bar\Fbb_p}$ is the affine Grassmannian for $\GL_2$ parametrising $\bar\Fbb_p[[u]]$-lattices in $\bar\Fbb_p((u))^2$.
We fix the isomorphism in $(\ref{prodisom})$ and write 
\begin{equation}\label{proji}
\pr_i:G_{\bar\Fbb_p}\longrightarrow (\GL_2)_{\bar\Fbb_p}
\end{equation}
for the projection to the $i$-th factor. We write $\langle -,-\rangle$ for the canonical pairing between characters and cocharacters and fix the characters
\begin{align*}
\dete:&\ \GL_2\longrightarrow \Gbb_m\\
\alpha:&\ T\longrightarrow \Gbb_m,
\end{align*}
where $\dete$ is the usual determinant, $T\subset \GL_2$ is the maximal torus of diagonal matrices and $\alpha$ is the unique dominant root defined by the Borel subgroup of upper triangular matrices, i.e. the character ${\rm diag}(t_1,t_2)\mapsto t_1t_2^{-1}$.\\
Let 
\begin{equation}\label{cocharnu}
\nu:(\Gbb_m)_{\bar\Fbb_p}\longrightarrow G_{\bar\Fbb_p}\cong \prod_{i=1}^n(\GL_2)_{\bar\Fbb_p}
\end{equation}
be a dominant cocharacter defined over the reflex field $\Fbb\subset \bar\Fbb_p$ and fix an $\Fbb$-valued point $A\in LG(\Fbb)=G(\Fbb((u)))$.
Given an $\Fbb_p$-algebra $R$ we write 
\[\phi:(R\otimes_{\Fbb_p}k)((u))\rightarrow (R\otimes_{\Fbb_p}k)((u))\]
for the homomorphism
that is the identity on $R$, the absolute Frobenius on $k$ and that maps $u$ to $u^p$.\\
We also write $\phi$ for the homomorphism $\bar\Fbb_p((u))\rightarrow \bar\Fbb_p((u))$ that is the identity on $\bar\Fbb_p$ and maps $u$ to $u^p$. It will always be clear from the context which $\phi$ is used.
By the construction in \cite[2.c.1]{phimod} there is a reduced projective $\Fbb$-variety
\[C_\nu(A)\subset \Fcal_G\otimes_{\Fbb_p}\Fbb\]
whose $\bar\Fbb_p$-valued points are given by 
\[\left\{ g(\bar\Fbb_p\otimes_{\Fbb_p}k)[[u]]^2\subset (\bar\Fbb_p\otimes_{\Fbb_p}k)((u))^2 \left|
{\begin{array}{*{20}c} 
g^{-1} A \phi(g)\in L^+G(\bar\Fbb_p)u^{\nu'}L^+G(\bar\Fbb_p) \\ \text{for some}\ \nu'\leq\nu.
\end{array}}
\right. \right\}\]
Here $\leq$ is the order on dominant coweights induced by our choice of the Borel and $u^{\nu'}$ is the image of $u\in\Gbb_m(\bar\Fbb_p((u)))$ under $\nu'$.
This variety is called the \emph{(closed) Kisin variety} associated with $\nu$ and $A$, compare \cite[6.a]{phimod}.\\
We will also use a different description of the closed points of $C_\nu(A)$. Let $N=(\bar\Fbb_p\otimes_{\Fbb_p}k)((u))^2$. We can consider the $\phi$-linear map
\begin{equation}\label{linsemilin}
\Phi_A=A(\id\otimes\phi):N\longrightarrow N.
\end{equation} 
We have an isomorphism
\[\begin{xy}\xymatrix{N\ar[r]^{\hspace{-10mm}\cong} & N_1\times\dots \times N_n}\end{xy},\]
where $N_i$ are $2$-dimensional $\bar\Fbb_p((u))$-vector spaces and the map $\Phi_A$ splits up into $\phi$-linear maps
\[\Phi_i:N_i\longrightarrow N_{i+1},\]
where we write $N_{n+1}=N_1$. \\
An $\bar\Fbb_p$-valued point of $\Fcal_G$ can be viewed as an $n$-tupel $\Mfrak=(\Mfrak_1,\dots,\Mfrak_n)$, where $\Mfrak_i\subset N_i$ is an $\bar\Fbb_p[[u]]$-lattice.\\
Given $\Mfrak=(\Mfrak_1,\dots,\Mfrak_n)\in\Fcal_G(\bar\Fbb_p)$ we denote by $(a_i,b_i)\in (\Z^2)_+$ the elementary divisors of $\Phi_{i-1}(\phi^\ast\Mfrak_{i-1})$ with respect to $\Mfrak_i$. \\ Consider the cocharacters $\mu_i(\Mfrak):\Gbb_m\rightarrow \GL_2$ defined by
\[u\longmapsto \begin{pmatrix}u^{a_i} & 0\\ 0 & u^{b_i}\end{pmatrix}.\]
Then the alternative description of $C_\nu(A)(\bar\Fbb_p)$ is given as follows: for $\Mfrak\in\Fcal_G(\bar\Fbb_p)$ we have
\[\Mfrak\in C_\nu(A)(\bar\Fbb_p)\Longleftrightarrow(\mu_1(\Mfrak),\dots,\mu_n(\Mfrak))\leq \nu,\]
with componentwise partition ordering.
To analyse the connected components of the varieties $C_\nu(A)$, we will use the language of Bruhat-Tits buildings (see \cite{Tits} for example). \\
Let $Z(G)$ denote the center of $G$ and write $\bar G=G/Z(G)$.\\
We write $\Bcal=\Bcal(LG(\bar\Fbb_p))$ for the Bruhat-Tits building of $G(\bar\Fbb_p((u)))$ and
$\bar\Bcal=\Bcal(L\bar G(\bar\Fbb_p))$ for the building of the group $\bar G(\bar\Fbb_p((u)))$. Then
\begin{align*}
\Bcal&\cong \Bcal_1\times\dots\times\Bcal_n\\
\bar\Bcal&\cong \bar\Bcal_1\times\dots\times\bar\Bcal_n,
\end{align*}
where $\Bcal_i$ is isomorphic to the building of $\GL_2(\bar\Fbb_p((u)))$ and $\bar\Bcal_i$ is isomorphic to the building of $\PGL_2(\bar\Fbb_p((u)))$, i.e. $\bar\Bcal_i$ is a topological space that is isomorphic to a tree where the link of every vertex is parametrized by $\mathbb{P}^1(\bar\Fbb_p)$. We write 
\[d_i:\bar\Bcal_i\times\bar\Bcal_i\longrightarrow \R\]
for the Weyl equivariant distance in the tree $\bar\Bcal_i$, normalized such that the distance of two neighbouring vertices is equal to $1$.
For a $\bar\Fbb_p[[u]]$-lattice $\Mfrak_i\subset N_i\cong \bar\Fbb_p((u))^2$ we write $\bar\Mfrak_i\in\bar\Bcal_i$ or sometimes $[\Mfrak_i]$ for its homothety class.
\begin{prop}\label{Abbinbuilding}
The map $\bar\Mfrak_i\mapsto [\Phi_i(\phi^\ast \Mfrak_i)]$ extends to a map $\bar\Phi_i:\bar\Bcal_i\rightarrow \bar\Bcal_{i+1}$. This map has the following properties:\\
\noindent {\rm (i)} $\bar\Phi_i$ takes geodesics to geodesics.\\
\noindent {\rm (ii)} Let $x,y\in\bar\Bcal_i$, then $d_{i+1}(\bar\Phi_i(x),\bar\Phi_i(y))=pd_i(x,y)$.
\end{prop}
\begin{proof}
The proof is the same as in \cite[6.b.2, 6.b.3]{phimod}.
\end{proof}
We want to reformulate the definition of the Kisin variety $C_\nu(A)$ in terms of the building. This is done in the following Proposition.
\begin{prop}\label{ptsinbuilding}
Let $\nu$ be a cocharacter as in $(\ref{cocharnu})$, defined over the reflex field $\Fbb$ and fix a point $A\in LG(\Fbb)$. Define the following integers:
\begin{align*}
r_i&=\langle \alpha\circ\pr_i,\nu\rangle\ ,\\
m_i&=\langle \dete\circ\pr_i,\nu\rangle\ .
\end{align*}
Let $\Mfrak=(\Mfrak_1,\dots,\Mfrak_n)\in\Fcal_G(\bar\Fbb_p)$ be an $n$-tuple of lattices. Then
\[\Mfrak\in C_\nu(A)(\bar\Fbb_p)\Longleftrightarrow  \begin{cases}
d_i(\bar\Mfrak_i,\bar\Phi_{i-1}(\bar\Mfrak_{i-1}))\leq r_i \\ 
\dete\ \Mfrak_i=u^{m_i}\dete\ \Phi_{i-1}(\phi^\ast\Mfrak_{i-1}). 
\end{cases}\]
Here in the last identity we mean equality of lattices in $\dete\ N_i\cong\bar\Fbb_p((u))$.
\end{prop}
\begin{proof}
Let $\Mfrak=g(\bar\Fbb\otimes_{\Fbb_p}k)[[u]]^2\in C_\nu(A)(\bar\Fbb_p)$. Write $\nu'$ for the cocharacter such that
\begin{equation}\label{konjbed}
g^{-1}A\phi(g)\in L^+G(\bar\Fbb)u^{\nu'}L^+G(\bar\Fbb).
\end{equation}
Under the isomorphism $(\ref{prodisom})$ we set $g=(g_1,\dots,g_n)$ and $A=(A_1,\dots,A_n)$, where $g_i\in \GL_2(\bar\Fbb((u)))$ and $A_i\in \GL_2(\bar\Fbb((u)))$.
Further the cocharacter $\nu'$ is given by an $n$-tupel $(\nu'_1,\dots,\nu'_n)$, where $\nu'_i:\Gbb_m\rightarrow \GL_2$ is a cocharacter.\\
Then $\Mfrak=(\Mfrak_1,\dots,\Mfrak_n)$ with $\Mfrak_i=g_i \bar\Fbb[[u]]^2\subset N_i$ and $(\ref{konjbed})$ translates to 
\begin{equation}\label{konjbed2}
g_{i+1}^{-1}A_{i+1}\phi(g_i)\in \GL_2(\bar\Fbb[[u]])u^{\nu'_{i+1}} \GL_2(\bar\Fbb[[u]])
\end{equation}  
for $i=1,\dots,n$ where we again identify $g_{n+1}=g_1$ and $A_{n+1}=A_1$.\\
Suppose that the cocharacter $\nu'_i$ is given by
\[u\longmapsto \begin{pmatrix} u^{a_i} & 0 \\ 0 & u^{b_i}\end{pmatrix};\]
then $(\ref{konjbed2})$ means that the elementary divisors of $\Phi_i(\phi^\ast\Mfrak_i)=A_{i+1}\phi(g_i)\bar\Fbb_p[[u]]^2$ with respect to $\Mfrak_{i+1}=g_{i+1}\bar\Fbb_p[[u]]^2$ are given by $(a_{i+1},b_{i+1})$. Hence $(\ref{konjbed2})$ is equivalent to 
\begin{align*}
d_{i+1}(\bar\Mfrak_{i+1},\bar\Phi_i(\bar\Mfrak_i))=a_{i+1}-b_{i+1}\\
\dete\ \Phi_i(\phi^\ast\Mfrak_i)=u^{a_{i+1}+b_{i+1}}\dete\ \Mfrak_{i+1}.
\end{align*}
Now the claim follows, as $\nu'\leq\nu$ is equivalent to
\begin{align*}
a_i+b_i&=m_i=\langle\dete\circ\pr_i,\nu\rangle\\
a_i-b_i&\leq r_i=\langle\alpha\circ\pr_i,\nu\rangle,& \text{for all}\ i=1,\dots,n.
\end{align*}
\end{proof}
\section{Some Lemmas in the building}
Before we proceed with the investigation of the Kisin variety, we need two Lemmas about some subvarietes of the affine Grassmannian defined in terms of the building. 
\begin{lem}\label{lemSchubertvar}
Let $m\in\Z$ and $x_1,x_2\in\Bcal(L(\PGL_2)(\bar\Fbb_p))$. Write $d$ for the distance function on $\Bcal(L(\PGL_2)(\bar\Fbb_p))$.
For $r_1,r_2\in\Q$, the closed subvariety 
\[X\subset \Fcal_{\GL_2}=L\GL_2/L^+\GL_2\]
whose closed points are given by
\[X(\bar\Fbb_p)= \left\{ \Mfrak \subset \bar\Fbb_p((u))^2 \left| 
{\begin{array}{*{20}c} 
\dete\ \Mfrak=u^m\bar\Fbb_p[[u]]\subset \bar\Fbb_p((u)) \\ 
d(\bar\Mfrak,x_1)\leq r_1 \\ 
d(\bar\Mfrak,x_2)\leq r_2 
\end{array}} \right.
\right\}
\]
is isomorphic to a Schubert variety in $\Fcal_{\GL_2}$, i.e it is the closure of some $\GL(\mathfrak{N})$ orbit for a suitable lattice $\Nfrak\subset \bar\Fbb_p((u))^2$.
Especially $X$ is connected.
\end{lem}
\begin{proof}
Denote by $[x_1,x_2]\subset \Bcal(\PGL_2(\bar\Fbb_p((u))))$ the geodesic between $x_1$ and $x_2$. We assume that $d(x_1,x_2)\leq r_1+r_2$, as otherwise $X$ is empty. 
Let $y\in[x_1,x_2]$ be the unique point with 
\[d(y,x_1)=\tfrac{1}{2}(d(x_1,x_2)+r_1-r_2).\]
Further we set
\[R=\tfrac{1}{2}(r_1+r_2-d(x_1,x_2)).\]
Then
\begin{equation}\label{onecirc}
X(\bar\Fbb_p)= \left\{
\Mfrak\subset \bar\Fbb_p((u))^2 \left|
{\begin{array}{*{20}c}
\dete\ \Mfrak=u^m\bar\Fbb[[u]]\\
d(\bar\Mfrak,y)\leq R
\end{array}}\right.
\right\},
\end{equation}
which can be seen as follows:\\
If $d(\bar \Mfrak,y)\leq R$, then 
\begin{align*}
d(\bar\Mfrak,x_1) &\leq d(\bar \Mfrak,y)+d(y,x_1)\leq R+d(y,x_1)=r_1
\end{align*}
and similarly for $x_2$. If conversely $d(\bar\Mfrak,x_i)\leq r_i$, then we choose an apartment containing $x_1$, $x_2$ and $y$. Denote by $z$ the projection of $\bar\Mfrak$ to this apartment and assume that $z$ is contained in the half line starting at $y$ and not containing $x_2$. Then
\[d(\bar\Mfrak,y)=d(\bar\Mfrak,x_2)-d(y,x_2)\leq r_2-d(y,x_2)=R.\]

We may replace $R$ by the smallest number such that the equality $(\ref{onecirc})$ is possible. 
If $y$ is a vertex, then we are done.

Assume this is not the case and denote by $z_1$ and $z_2$ the endpoints of the edge containing $y$.
Fix a lattice $\Mfrak_0\in X(\bar\Fbb_p)$ such that $d(\bar\Mfrak_0,y)=R$. Such a lattice exists by our definition of $R$. Then we take for $z$ the unique element of $\{z_1,z_2\}$ such that $z\in[\bar\Mfrak_0,y]$. If we define 
\begin{equation*}\label{tildeR}
\widetilde{R}=\max\{d(z,\bar\Mfrak)\mid \Mfrak\in X(\bar\Fbb_p)\},
\end{equation*}
then we easily see that $\widetilde{R}=R-d(z,y)$. We now have
\begin{equation*}\label{Xpts}
X(\bar\Fbb_p)=\left\{ \Mfrak\left| 
{\begin{array}{*{20}c}
\dete\ \Mfrak=u^m\bar\Fbb_p[[u]] \\
d(\bar\Mfrak,z)\leq \widetilde{R}
\end{array}}\right.
\right\}.\end{equation*}
The inclusion "$\subset$" is obvious from the definition of $\widetilde{R}$. The converse inclusion follows from 
\[d(\bar\Mfrak,z)=\begin{cases}
d(\bar\Mfrak,y)-d(z,y) & \text{if}\ y\notin[\bar\Mfrak,z] \\
d(\bar\Mfrak,y)+d(z,y) & \text{if}\ y\in[\bar\Mfrak,z],
\end{cases}\]
the definition of $\widetilde{R}$ and the fact that the distance between the homothety classes of two lattices with equal determinant is even.

It now follows that $X$ is the closure of the $\GL(\mathfrak{N})$-orbit of $\widetilde{\Mfrak}$,
where $\mathfrak{N}$ is some lattice with homothety class $z$ and $\widetilde{\Mfrak}$ is some lattice with $\dete\ \widetilde{\Mfrak}=u^m\bar\Fbb_p[[u]]$ and $d([\widetilde{\Mfrak}],[\mathfrak{N}])=\widetilde{R}$.
\end{proof}
\begin{prop}\label{propconnected}
Let $s\geq 3$ and let $N_1,\dots,N_s$ be $2$-dimensional $\bar\Fbb_p((u))$-vector spaces together with fixed isomorphisms with $\bar\Fbb_p((u))^2$, and let $\Fcal_i$ be the affine Grassmannian of $N_i$. Suppose that there are $\phi$-linear maps 
\[\Phi_i:N_i\longrightarrow N_{i+1}\]
such that their linearisations are isomorphisms.
Let $\Mfrak_1\subset N_1$ and $\Mfrak_s\subset N_s$ be lattices. Further we fix $r_2,\dots,r_s\in\mathbb{Q}$ and $m_2,\dots,m_{s-1}\in\Z$. Then the closed subvariety $X\subset \Fcal_2\times\dots\times\Fcal_{s-1}$ defined by
\[X(\bar\Fbb_p)=\left\{(\Mfrak_2,\dots,\Mfrak_{s-1})\in\Fcal_2\times\dots\times\Fcal_{s-1}\left|
{\begin{array}{*{20}c}
\dete\ \Mfrak_i=u^{m_i}\bar\Fbb_p[[u]]\\
d_i(\bar\Phi_{i-1}(\bar\Mfrak_{i-1}),\bar\Mfrak_i)\leq r_i \\ \text{for}\ i=2,\dots,s
\end{array}}\right.
\right\}\]
is connected. 
\end{prop}
\begin{proof}
We proceed by induction. \\
Let $s=3$. Then we deduce the claim from Lemma $\ref{lemSchubertvar}$ as follows:
The $\bar\Fbb_p$-valued points of the subvariety $X\subset \Fcal_2$ are the lattices $\Mfrak_2\subset N_2$ such that $\dete\ \Mfrak_2=u^{m_2}\bar\Fbb_p[[u]]$ and 
\begin{align*}
d_2(\Phi_1(\bar\Mfrak_1),\Mfrak_2)&\leq r_2\\
d_2(\bar\Mfrak_2,\bar\Phi_2^{-1}(\bar\Mfrak_3))&\leq r_3/p,
\end{align*}
where we write $\bar\Phi_2^{-1}(\bar\Mfrak_3)$ for the preimage of $\bar\Mfrak_3$. This does not necessarily lie in the building $\bar\Bcal_2=\Bcal(\PGL(N_2))$, but is visible after some ramified extension of $\bar\Fbb_p((u))$. Replacing $\bar\Phi_2^{-1}(\bar\Mfrak_3)$ by its projection to $\bar\Bcal_2$ and $r_3/p$ by 
\[r_3/p-d_2(\bar\Phi_2^{-1}(\bar\Mfrak_3),\bar\Bcal_2),\]
the case $s=3$ now is a consequence of Lemma $\ref{lemSchubertvar}$.

Assume that the assertion holds for $s-1$. Then the fibers of $X$ over every point in $\Fcal_2$ are connected by induction hypothesis. And hence $X$ is connected if and only if $\pr_2(X)\subset \Fcal_2$ is connected. Here we use the properness of the projection $\pr_2$.\\
Let $y_{s-1}$ be the projection of $\bar\Mfrak_s$ onto the convex subset $\bar\Phi_{s-1}(\bar\Bcal_{s-1})$. This is indeed a convex subset as $\bar\Phi_{s-1}$ is injective and maps an apartment in $\bar\Bcal_{s-1}$ onto an apartment in $\bar\Bcal_s$. If we set $t_s=d_s(\bar\Mfrak_s,y_{s-1})$, then
\begin{align*}
d_s(\bar\Phi_{s-1}(z),\bar\Mfrak_s)\leq r_s & \Leftrightarrow d_s(\bar\Phi_{s-1}(z),y_{s-1})\leq r_s-t_s\\
&\Leftrightarrow d_{s-1}(z,\bar\Phi_{s-1}^{-1}(y_{s-1}))\leq \tfrac{1}{p}(r_s-t_s),
\end{align*}
for all $z\in\bar\Bcal_{s-1}$. 
Inductively we denote by $y_{s-i}$ the projection of $y_{s-i+1}\in\bar\Bcal_s$ onto the convex subset $(\bar\Phi_{s-1}\circ\dots\circ\bar\Phi_{s-i})(\bar\Bcal_{s-i})\subset \bar\Bcal_s$ and write $t_{s-i+1}$
for the integer defined by
\[d_s(y_{s-i+1},y_{s-i})=p^{i-1}t_{s-i+1}.\]
This is indeed an integer as $d_{s-i+1}(\bar\Phi_{s-i}(x),y)$ is an integer for all vertices $x\in\bar\Bcal_{s-i}$ and $y\in\bar\Bcal_{s-i+1}$ (and $\Phi_j$ multiplies distances by $p$). Now we have
\[\pr_2(X(\bar\Fbb_p))=\left\{
\Mfrak_2\subset N_2 \left| 
{\begin{array}{*{20}c}
\dete\ \Mfrak_2=u^{m_2}\bar\Fbb_p[[u]]\\
d_2(\bar\Mfrak_2,\bar\Phi_1(\bar\Mfrak_1))\leq r_2 \\
d_2(\bar\Mfrak_2,(\bar\Phi_{s-1}\circ\dots\circ\bar\Phi_2)^{-1}(y_2))\leq R/p^{s-2}
\end{array}}\right.
\right\},\]
where $R=(r_s-t_s)+p(r_{s-1}-t_{s-1})+\dots+p^{s-3}(r_3-t_3)$. 
This can be seen as follows: Inductively (starting with the observation above) we can show that for any lattice $\Mfrak_{s-i}\subset N_{s-i}$ we have
\[d_{s-i}(\bar\Mfrak_{s-i},(\bar\Phi_s\circ\dots\circ\bar\Phi_{s-i})^{-1}(y_{s-i}))\leq \frac{1}{p^i}\sum_{j=0}^{i-1} p^j(r_{s-j}-t_{s-j})\]
if and only if there exist an $\bar\Mfrak_{s-i+1}$ such that 
\begin{align*}
d_{s-i+1}(\bar\Mfrak_{s-i+1},\bar\Phi_{s-i}(\bar\Mfrak_{s-i}))& \leq r_{s-i+1} \\ 
d_{s-i}(\bar\Mfrak_{s-i+1},(\bar\Phi_s\circ\dots\circ\bar\Phi_{s-i+1})^{-1}(y_{s-i+1}))&\leq \frac{1}{p^{i-1}}\sum_{j=0}^{i-2} p^{j}(r_{s-j}-t_{s-j}).
\end{align*}
Now the claim follows from Lemma $\ref{lemSchubertvar}$.
\end{proof}
We return to the setting of the first section. Let $\nu$ be a cocharacter as in $(\ref{cocharnu})$ and define 
\begin{align*}
m_i&=\langle\dete\circ\pr_i,\nu\rangle,\\
r_i&=\langle\alpha\circ\pr_i,\nu\rangle.
\end{align*}
We define a subset 
\begin{equation}\label{Bcalnu}
\Bcal(\nu)=\Bcal_1(\nu)\times\dots\times\Bcal_n(\nu)\subset \Bcal,
\end{equation}
where $\Bcal_i(\nu)$ is the convex hull of all lattices $\Mfrak_i\subset N_i$
such that $\dete\ \Mfrak_i=u^{s_i}\bar\Fbb_p[[u]]$, where the $s_i$ are integers defined by the set of equations 
\begin{align*}
p^2s_i+m_{i+1}&=s_{i+1} & \text{for}\ i=1,\dots,n-1\\
p^2s_n+m_1&=s_1.
\end{align*}
\begin{rem}\label{remdisteven}
Note that $\Bcal(\nu)$ maps homeomorphically onto $\bar\Bcal$. But not every vertex of $\Bcal(\nu)$ is defined by a lattice. More precisely, if $\Mfrak_i$ and $\mathfrak{N}_i$ are lattices in $\Bcal_i(\nu)$, then $d_i(\bar\Mfrak_i,\bar{\mathfrak{N}}_i)$ is even.
\end{rem} 
The integers $s_i$ are defined in a way such that $C_\nu(A)(\bar\Fbb_p)\subset \Bcal(\nu)$:
If $\Mfrak=(\Mfrak_1,\dots,\Mfrak_n)\in C_\nu(A)(\bar\Fbb_p)$ with $\dete\ \Mfrak_i=u^{s_i}\bar\Fbb_p[[u]]$, then
\[\dete\ \Mfrak_{i+1}=u^{m_{i+1}}\dete\ \Phi(\phi^\ast\Mfrak_i)\]
implies $s_{i+1}=m_{i+1}+p^2s_i$.
\begin{cor}\label{Corconnected}
Denote by $\pr_i:C_\nu(A)\rightarrow \Fcal_i$ the projection onto the $i$-th factor in $(\ref{proji})$. Then the fibers of $\pr_i$ are connected.
\end{cor}
\begin{proof}
With the notations from above we find $C_\nu(A)(\bar\Fbb_p)\subset\Bcal(\nu)$.
Then the claim follows from the description of the closed points in Proposition $\ref{ptsinbuilding}$ and Proposition $\ref{propconnected}$. 
\end{proof}
\section{The simple case}
In this section we will prove the following theorem.
\begin{theo}\label{theoirred}
Let $A\in LG(\bar\Fbb_p)$ and assume that the object 
\[(N,\Phi_A)=((\bar\Fbb_p\otimes_{\Fbb_p}k)((u))^2,\Phi_A)\]
is simple, where $\Phi_A$ is defined as in $(\ref{linsemilin})$. Let $\nu$ be a cocharacter as in $(\ref{cocharnu})$. Then the Kisin variety $C_\nu(A)$ is geometrically connected. 
\end{theo}
The idea of the proof is to analyse the set of lattices in the set $\Bcal(\nu)$ defined in $(\ref{Bcalnu})$ that correspond to closed points in $C_\nu(A)$. In the following we will write $[x,y]$
for the geodesic between two points $x,y\in\bar\Bcal$ (resp. $x,y\in\bar\Bcal_i$).
\begin{prop}\label{fixpt}
Let $A\in LG(\bar\Fbb_p)$ and assume that the object 
\[(N,\Phi_A)=((\bar\Fbb_p\otimes k)((u))^2,\Phi_A)\] defined by $(\ref{linsemilin})$ is simple, i.e. there is no proper $\Phi_A$-stable subspace.\\
\noindent {\rm (i)} There is a unique point $P=(P_1,\dots,P_n)\in\bar\Bcal=\bar\Bcal_1\times\dots\times\bar\Bcal_n$ fixed by the induced map $\bar\Phi_A$.\\
\noindent {\rm (ii)} For all $i\in\{1,\dots,n\}$ the point $P_i\in\bar\Bcal_i$ is not a vertex.\\
\noindent {\rm (iii)} Let $\Mfrak=(\Mfrak_1,\dots,\Mfrak_n)\in\Fcal_G(\bar\Fbb_p)$. Then there is $i\in\{1,\dots,n\}$ such that 
\[P_i\in[\bar\Mfrak_i,\bar\Phi_{i-1}(\bar\Mfrak_{i-1})].\]
\end{prop}
\begin{proof}
\noindent (i) Again we write 
\[N=(\bar\Fbb_p\otimes k)((u))^2=N_1\times\dots\times N_n\]
and $\Phi_i=\Phi_A|_{N_i}:N_i\rightarrow N_{i+1}$. Then
\[\widetilde{\Phi}:=\Phi_n\circ\dots\circ\Phi_1=\Phi_A^n|_{N_1}:N_1\longrightarrow N_1\]
is a $\widetilde{\phi}$-linear endomorphism of $N_1$ whose linearisation is an isomorphism, where $\widetilde{\phi}:\bar\Fbb_p((u))\rightarrow \bar\Fbb_p((u))$
is the identity on $\bar\Fbb_p$ and maps $u$ to $u^{p^n}$.\\
Further $((\bar\Fbb_p\otimes k)((u))^2,\Phi_A)$ is simple if and only if $(N_1,\widetilde{\Phi})$ is. By Caruso's classification of simple objects, see \cite[Corollary 8]{Caruso}, there is an $\bar\Fbb_p((u))$-basis $b_1,b_2$ of $N_1$ such that
\begin{align*}
& \widetilde{\Phi}(b_1)=b_2 \\
& \widetilde{\Phi}(b_2)=a u^sb_1
\end{align*}
for some $a\in\bar\Fbb^\times_p$ and $s\in\Z$. Let $\Acal\subset \bar\Bcal_1$ be the apartment defined by $b_1,b_2$ and let $\Acal\cong \R$ be an isomorphism preserving the distance and under which the homothety class of $\bar\Fbb[[u]] b_1\oplus\bar\Fbb_p[[u]]b_2$ is mapped to $0$. Then one easily calculates that the preimage of $s/(p^n+1)$ is fixed under 
\[[\widetilde{\Phi}]=\overline{\widetilde{\Phi}}:\bar\Bcal_1\longrightarrow \bar\Bcal_1.\]
Denote this fixed point by $P_1$ and inductively define 
\[P_i=\bar\Phi_{i-1}(P_{i-1})\in\bar\Bcal_i\]
 for $i=2,\dots,n$. Then $P=(P_1,\dots,P_n)$ is fixed by $\bar\Phi_A$.
Further it is unique because of Proposition $\ref{Abbinbuilding}$, (ii). 

\noindent (ii) As the isomorphism $\Acal\cong \R$ maps the vertices of $\Acal$ exactly to the integers, the point $P_1$ is a vertex if and only if $p^n+1 | s$.
In  this case one easily checks that 
\[\bar\Fbb_p((u))(\sqrt{a}u^{s/(p^n+1)}b_1+b_2)\]
is a $\widetilde{\Phi}$-stable subspace. Contradiction. Hence $P_1$ is not a vertex.\\
Instead of constructing $P_i$ using $P_1$ we could also have used 
\[\Phi_A^n|_{N_i}:N_i\longrightarrow N_i\]
and hence the same argument shows that $P_i$ is not a vertex.

\noindent (iii) We first show that $P_1\in [\bar\Mfrak_1,[\widetilde{\Phi}](\bar\Mfrak_1)]$. Using the above notation we denote by $x_1$ and $x_2$ the vertices of $\bar\Bcal_1$ such that $P_1\in[x_1,x_2]$ and $d_1(x_1,x_2)=1$. As $\bar\Phi^n$ maps geodesics to geodesics it is enough to check that $x_1\in[P_1,\bar\Phi^n(x_2)]$ and vice versa. Using the standard form for $\Phi^n$ from above this is an easy computation.

Now assume $P_i\notin[\Mfrak_i,\bar\Phi_{i-1}(\bar\Mfrak_{i-1})]$ for all indices $i$. Then $P_2\notin [\bar\Mfrak_2,\bar\Phi_1(\bar\Mfrak_1)]$ implies
\[P_3=\bar\Phi_2(P_2)\notin [\bar\Phi_2(\bar\Mfrak_2),\bar\Phi_2(\bar\Phi_1(\bar\Mfrak_1))]\]
Together with $P_3\notin [\bar\Mfrak_3,\bar\Phi_2(\bar\Mfrak_2)]$ this implies $P_3\notin[\bar\Mfrak_3,\bar\Phi_2(\bar\Phi_1(\bar\Mfrak_1))]$, as 
\[[\bar\Mfrak_3,\bar\Phi_2(\bar\Phi_1(\bar\Mfrak_1))]\subset [\bar\Mfrak_3,\bar\Phi_2(\bar\Mfrak_2)]\cup[\bar\Phi_2(\bar\Mfrak_2),\bar\Phi_2(\bar\Phi_1(\bar\Mfrak_1))].\] 
Proceeding by induction this implies 
\[P_1\notin [\bar\Mfrak_1,\overline{\widetilde{\Phi}}(\bar\Mfrak_1)]=[\bar\Mfrak_1,(\bar\Phi_n\circ\dots\circ\bar\Phi_1)(\bar\Mfrak_1)],\]
contradicting the above.
\end{proof}
Given a cocharacter $\nu$ as in $(\ref{cocharnu})$, recall the definition of
\[\Bcal(\nu)=\Bcal_1(\nu)\times\dots\times\Bcal_n(\nu)\] 
from $(\ref{Bcalnu})$. This set maps homeomorphically onto $\bar\Bcal$.
For $i\in\{1,\dots,n\}$ we write $Q_i\in\bar\Bcal_i$ for the unique vertex with minimal distance from $P_i$ that is the homothety class of some lattice in $\Bcal_i(\nu)$.
This vertex is indeed unique as $P_i$ is not a vertex, and the distance between two homothety classes of lattices in $\Bcal_i(\nu)$ is always even, compare Remark $\ref{remdisteven}$.
 
We construct a lattice $\Mfrak(Q_i)\in\Bcal(\nu)$ as follows:
Let $\Mfrak(Q_i)_i\in\Bcal_i(\nu)$ be the unique lattice such that $[\Mfrak(Q_i)_i]=Q_i$. Then $\Mfrak(Q_i)_j$ is defined inductively for $j\neq i$ as the unique lattice in 
\[\left\{
\mathfrak{N}\subset N_j\left| 
\begin{array}{*{20}c}
 \mathfrak{N}\in\Bcal_j(\nu) \\
{} [\mathfrak{N}]\in [Q_j,\bar\Phi_{j-1}([\Mfrak(Q_i)_{j-1}])]\\
d_j([\mathfrak{N}],\bar\Phi_{j-1}([\Mfrak(Q_i)_{j-1}]))\leq r_j
\end{array}\right.
\right\}\]
with minimal distance from $Q_j$. 
\begin{lem}\label{MQiinC}
Let $i\in\{1,\dots,n\}$. If there exists $\Mfrak\in C_\nu(A)(\bar\Fbb_p)$ with $\bar\Mfrak_i=Q_i$, then $\Mfrak(Q_i)\in C_\nu(A)$.
\end{lem}
\begin{proof}
We only need to check
\[d_i(\bar\Phi_{i-1}[\Mfrak(Q_i)_{i-1}],Q_i)\leq r_i.\]
By induction one easily sees that
\[d_j(\bar\Phi_{j-1}([\Mfrak(Q_i)_{j-1}]),Q_j)\leq d_j(\bar\Phi_{j-1}([\Mfrak_{j-1}]),Q_j),\]
for all $j$, using that this is true for $P_j$ instead of $Q_j$ and the fact that the difference between the left hand side and the right hand side is even.
The claim now follows from 
\[d_i(\bar\Phi_{i-1}(\Mfrak_{i-1}),Q_i)=d_i(\bar\Phi_{i-1}(\Mfrak_{i-1}),\Mfrak_i)\leq r_i.\]
\end{proof}
\begin{lem}\label{MQiMQj}
Let $i,j\in\{1,\dots,n\}$ such that $\Mfrak(Q_i),\Mfrak(Q_j)\in C_\nu(A)$. Then $\Mfrak(Q_i)_j=\Mfrak(Q_j)_j$ or $\Mfrak(Q_j)_i=\Mfrak(Q_i)_i$.
\end{lem}  
\begin{proof}
Assume $j<i$ and $[\Mfrak(Q_i)_j]\neq Q_j$. Then the constructions of $\Mfrak(Q_i)$ and $\Mfrak(Q_j)$ imply 
\[d_{j+1}(Q_{j+1},[\Mfrak(Q_j)_{j+1}])\leq d_{j+1}(Q_{j+1},[\Mfrak(Q_i)_{j+1}]).\]
Proceeding by induction the constructions yield 
\[d_i(Q_i,[\Mfrak(Q_j)_i])\leq d_i(Q_i,[\Mfrak(Q_i)_i])=0\]
and hence $\Mfrak(Q_i)_i=\Mfrak(Q_j)_i$.
\end{proof}
\begin{lem}\label{P1nachGrass}
Let $\mathfrak{N}_1,\mathfrak{N}_2\subset \bar\Fbb_p((u))^2$ be lattices with $\dete\ \mathfrak{N}_1=\dete\ \mathfrak{N}_2$. 
Let $y$ denote the midpoint of the geodesic \[[\bar{\mathfrak{N}}_1,\bar{\mathfrak{N}}_2]\subset \Bcal(\PGL_2(\bar\Fbb_p((u)))).\]
There exists a morphism
\[\chi=\chi_{\bar{\mathfrak{N}}_1,\bar{\mathfrak{N}}_2}:\Pbb^1\longrightarrow \Fcal_{\GL_2}\]
such that $\chi(0)=\mathfrak{N}_1$ and $\chi(\infty)=\mathfrak{N}_2$ and such that
\[d([\chi(z)],y)=\tfrac{1}{2}d(\bar{\mathfrak{N}}_1,\bar{\mathfrak{N}}_2)\]
for all $z\in\bar\Fbb_p$. Note that the right hand side is an integer.
\end{lem}
\begin{proof}
This is similar to \cite[Lemma 3.7]{ich}.
\end{proof}
\begin{prop}\label{mainstep}
Let $\Mfrak=(\Mfrak_1,\dots,\Mfrak_n)\in C_\nu(A)$. Then there exists an index $i\in\{1,\dots,n\}$ such that for all $\mathfrak{N}_i\subset N_i$ with $\dete\,\mathfrak{N}_i=\dete\,\Mfrak_i$ and $y_i\in[\bar P_i,\bar\Mfrak_i]$, where $y_i$ denotes the midpoint of $[\bar{\mathfrak{N}_i},\bar\Mfrak_i]$, there is a morphism
\[\chi:\Pbb^1_{\bar\Fbb_p}\longrightarrow \Fcal_i\]
such that $\chi(\Pbb^1)\subset \pr_i(C_\nu(A))$ and $\chi(0)=\Mfrak_i$ and $\chi(\infty)=\mathfrak{N}_i$.
\end{prop}
\begin{proof}
Let $i\in\{1,\dots,n\}$ such that $P_i\in[\bar\Mfrak_i,\bar\Phi_{i-1}(\bar\Mfrak_{i-1})]$. This index exists because of Proposition $\ref{fixpt}$, (ii). If $\bar\Mfrak_i=Q_i$, then we are done.
Assume this is not the case. 

With the notation from Lemma $\ref{P1nachGrass}$ we consider the morphism $\chi=\chi_{\bar\Mfrak_i,\bar{\mathfrak{N}}_i}$. Then $\chi(0)=\Mfrak_i$ and $\chi(\infty)=\mathfrak{N}_i$.
As $\pr_i(C_\nu(A))$ is closed it is sufficient to show that $\chi(\mathbb{A}^1(\bar\Fbb_p))\subset \pr_i(C_\nu(A)(\bar\Fbb_p))$. By construction we have for all $z\in\bar\Fbb_p$,
\[d_i([\chi(z)],y_i)=d_i([\chi(\infty)],y_i),\]
where $y_i$ denotes the midpoint of $[\bar{\mathfrak{N}}_i,\bar\Mfrak_i]$. As $y_i\in[P_i,\bar\Mfrak_i]$ we have
\begin{align*}
 d_i([\chi(z)],P_i) & =d_i([\Mfrak_i],P_i)\ ,\\
 d_i([\chi(z)],\bar\Phi_{i-1}([\Mfrak_{i-1}])) & = d_i([\chi(z)],P_i)+d_i(P_i,\bar\Phi_{i-1}([\Mfrak_{i-1}]))\\ 
&= d_i([\Mfrak_i],P_i)+d_i(P_i,\bar\Phi_{i-1}([\Mfrak_{i-1}]))\\
&=d_i([\Mfrak_i],\bar\Phi_{i-1}([\Mfrak_{i-1}]))
\end{align*}
for all $z\in\bar\Fbb_p$.
\begin{figure}\label{figur1} 
\par
\hspace{0in}
\begin{center}
\includegraphics[scale=0.4]{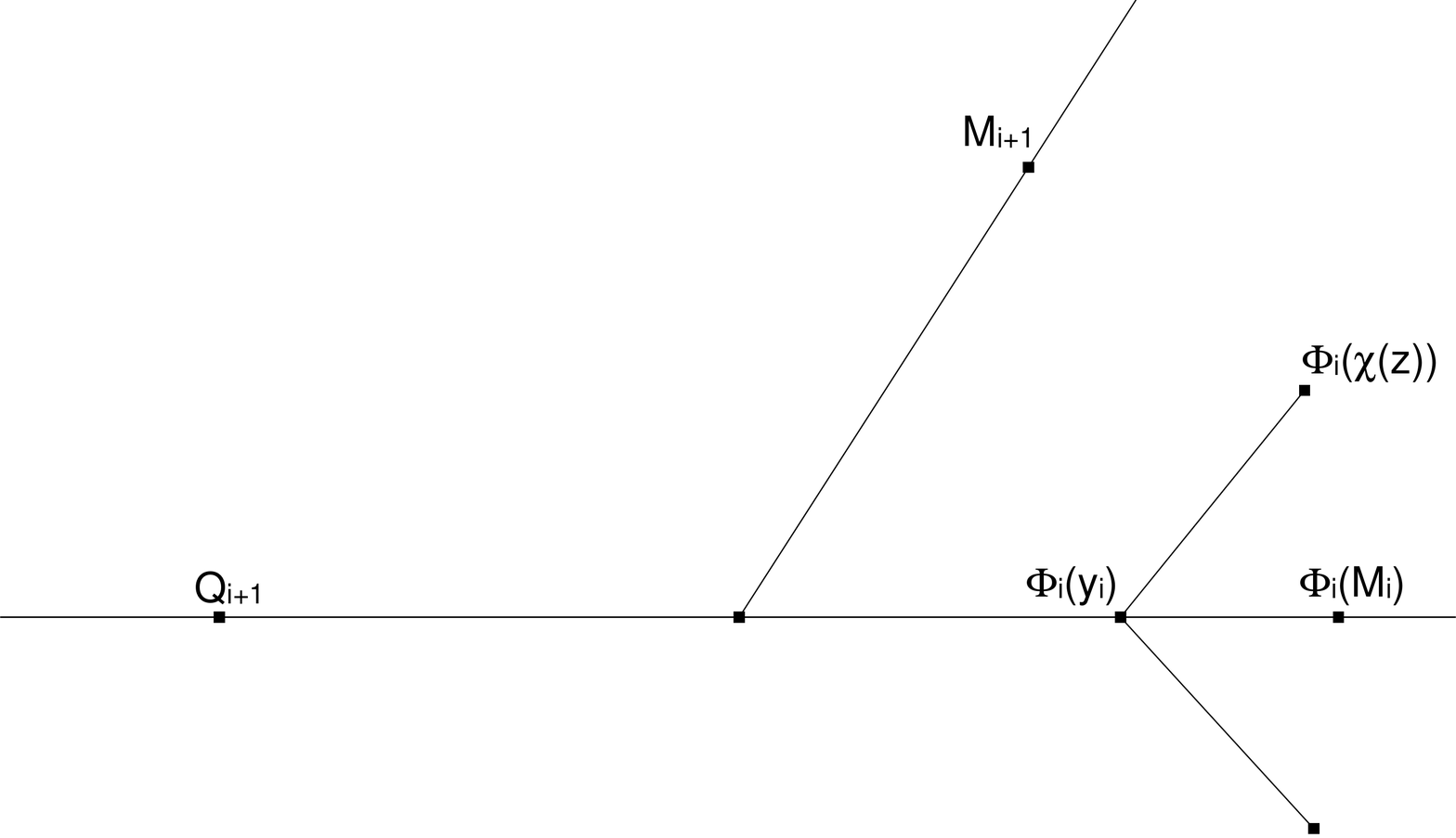}
\end{center}
\caption{The relative position of $\bar\Phi_i(\chi(z))$ and $\bar\Mfrak_{i+1}=M_{i+1}$ in the building.}
\end{figure}


\begin{lem}\label{geodesics}
If $\bar\Phi_i(y_i)\notin [Q_{i+1},\bar\Mfrak_{i+1}]$, then  $\bar\Phi_i(y_i)\in [\bar\Mfrak_{i+1},\bar\Phi_i([\chi(z)])]$ for all $z\in\bar\Fbb_p$.
\end{lem}
\begin{proof}
We can choose an apartment $\bar\Acal\subset \bar\Bcal_{i+1}$ containing the points $Q_{i+1}$, $P_{i+1}$, $\bar\Phi_i(Q_i)$, $\bar\Phi_i(y_i)$ and $\bar\Phi_i(\Mfrak_i)$.
We denote by $z$ the projection of $\bar\Mfrak_{i+1}$ onto this apartment. Denote by $\bar\Acal_{-}$ the half line in $\bar\Acal$ starting at $\bar\Phi_i(y_i)$ and containing $Q_{i+1}$. Our assumption means $z\in\bar\Acal_-$. But we also have $\bar\Phi_i(Q_i)\in\bar\Acal_-$ as the corresponding fact holds true for $P_i$ instead of $Q_i$. The claim now follows from
\[[\bar\Phi_i(Q_i),\bar\Phi_i(y_i)]\cap[\bar\Phi_i([\chi(z)]),\bar\Phi_i(y_i)]=\{\bar\Phi_i(y_i)\}.\] 
\end{proof}
Assume first that $\bar\Phi_i(y_i)\notin[Q_{i+1},\bar\Mfrak_{i+1}]$. By Lemma $\ref{geodesics}$ we find
\begin{align*}
d_{i+1}([\Mfrak_{i+1}],\bar\Phi_i([\chi(z)]))&= d_{i+1}([\Mfrak_{i+1}],\bar\Phi_i(y_i))+d_{i+1}(\bar\Phi_i(y_i),\bar\Phi([\chi(z)]))\\
&=d_{i+1}([\Mfrak_{i+1}],\bar\Phi_i(y_i))+d_{i+1}(\bar\Phi_i(y_i),\bar\Phi([\Mfrak_i]))\\
&=d_{i+1}([\Mfrak_{i+1}],\bar\Phi_i([\Mfrak_i]))\leq r_{i+1}
\end{align*}
for all $z\in\bar\Fbb_p$ (compare Fig. $1$).  
It follows that 
\[(\Mfrak_1,\dots,\Mfrak_{i-1},\chi(z),\Mfrak_{i+1},\dots,\Mfrak_n)\in C_\nu(A)\]
for all $z$ and hence we are done. 

If $\bar\Phi_i(y_i)\in [Q_{i+1},\bar\Mfrak_{i+1}]$, then consider the morphism
\[\chi_{i+1}=\chi_{\bar\Mfrak_{i+1},z_{i+1}}|_{\Abb^1}:\Abb^1\longrightarrow \Fcal_{i+1},\]
where $z_{i+1}$ is the point in $\bar\Bcal_{i+1}$ defined as follows. The point $\bar\Phi_i(y_i)$ is the midpoint of the geodesic $[z_{i+1},\bar\Mfrak_{i+1}]$ and there is an apartment $\bar\Acal\subset\bar\Bcal_{i+1}$ such that 
\[[z_{i+1},\bar\Mfrak_{i+1}],[Q_{i+1},\bar\Mfrak_{i+1}]\subset \bar\Acal,\]
compare also Fig. $2$. 
Then, by construction, we have $\bar\Phi_i(y_i)\in[[\chi(z)],[\chi_{i+1}(z)]]$ for all $z$ and hence
\begin{align*}
d_{i+1}([\chi_{i+1}(z)],\bar\Phi_i([\chi(z)])) &= d_{i+1}([\chi_{i+1}(z)],\bar\Phi_i(y_i))+d_{i+1}(\bar\Phi_i(y_i),\bar\Phi_i([\chi(z)]))\\
&=d_{i+1}([\chi_{i+1}(0)],\bar\Phi_i(y_i))+d_{i+1}(\bar\Phi_i(y_i),\bar\Phi_i([\chi(0)]))\\
&= d_{i+1}([\chi_{i+1}(0)],\bar\Phi_i([\chi(0)])); \\
d_{i+1}([\chi_{i+1}(z)],P_{i+1}) &= d_{i+1}([\chi_{i+1}(0)],P_{i+1})
\end{align*}
for all $z\in\bar\Fbb_p$ (compare Fig. $2$). 
\begin{figure}\label{figur2} 
\par
\hspace{0in}
\begin{center}
\includegraphics[scale=0.4]{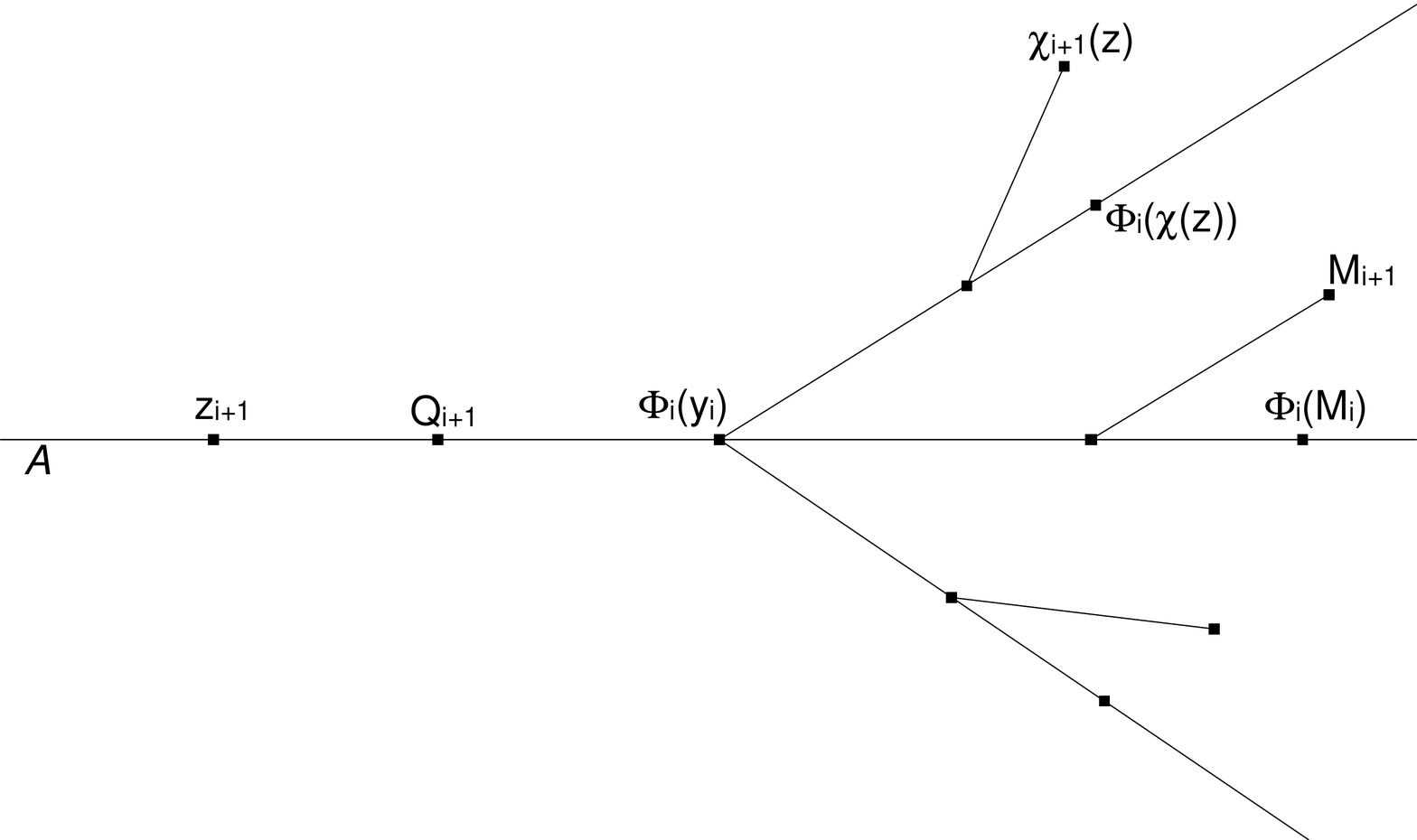}
\end{center}
\caption{The relative position of $\bar\Phi_i(\chi(z))$ and $\overline{\chi_{i+1}(z)}$ in the building.}
\end{figure}

Denote by $y_{i+1}=\bar\Phi_i(y_i)$ the midpoint of $[z_{i+1},\bar\Mfrak_{i+1}]$.
If $\bar\Phi_{i+1}(y_{i+1})\notin[Q_{i+2},\bar\Mfrak_{i+2}]$, then 
\[d_{i+2}([\Mfrak_{i+2}],\bar\Phi_{i-1}([\chi_{i+1}(z)]))=d_{i+2}([\Mfrak_{i+2}],\bar\Phi_{i-1}([\chi_{i+1}(0)]))\]
for all $z\in\bar\Fbb_p$ by the same argument as above and hence 
\[(\Mfrak_1,\dots,\Mfrak_{i-1},\chi(z),\chi_{i+1}(z),\Mfrak_{i+2},\dots,\Mfrak_{n})\in C_\nu(A),\]
and we are done. 

Otherwise, if $\bar\Phi_{i+1}(y_{i+1})\in[Q_{i+2},\bar\Mfrak_{i+2}]$, we proceed inductively as above. If this procedure does not stop we end up with maps
\[\chi_j:\Abb_{\bar\Fbb_p}^1\longrightarrow \Fcal_j\]
for $j=1,\dots,n$,
where $\chi_i=\chi|_{\Abb^1}$ and $\chi_j(0)=\Mfrak_j$ for all $j$. Further 
\begin{align*}
d([\chi_j(z)],P_j)&= d([\chi_j(0)],P_j) &&\text{for all}\ z\in\bar\Fbb_p\\
d([\chi_j(z)],\bar\Phi_{j-1}([\chi_{j-1}(z)])) &= d([\chi_j(0)],\bar\Phi_{j-1}([\chi_{j-1}(0)])) &&\text{for all}\ j\neq i.
\end{align*}
Finally our constructions imply that the midpoint of $[P_j,[\chi_j(z)]]$ is independent of $z\in\bar\Fbb_p$ and this implies (together with $P_i\in [\bar\Mfrak_i,\bar\Phi_{i-1}(\bar\Mfrak_{i-1})]$) that
\[d_i([\chi_i(z)],\bar\Phi_{i-1}([\chi_{i-1}(z)])) = d_i([\chi_i(0)],\bar\Phi_{i-1}([\chi_{i-1}(0)]))\leq r_i.\]
Hence we have $(\chi_1(z),\dots,\chi_n(z))\in C_\nu(A)$ for all $z\in\bar\Fbb_p$.
\end{proof}
\begin{cor}
There exists an index $i\in\{1,\dots,n\}$ such that the projection $\pr_i(C_\nu(A))\subset \Fcal_i$ is isomorphic to a Schubert variety.
\end{cor}
\begin{proof}
Let $i\in\{1,\dots,n\}$ be as in the Proposition above, i.e. such that there exists $\Mfrak\in C_\nu(A)(\bar\Fbb_p)$ such that $P_i\in[\bar\Mfrak_i,\bar\Phi_{i-1}(\bar\Mfrak_{i-1})]$.
We set
\[R=\max\{d_i(\bar\Mfrak_i,P_i)\mid \Mfrak=(\Mfrak_1,\dots,\Mfrak_n)\in C_\nu(A)(\bar\Fbb_p)\},\]
and fix an $\Mfrak_i$, where this maximum is obtained.
Applying Proposition $\ref{mainstep}$ several times we find that
\[\{\mathfrak{N}_i\subset N_i\mid \dete\,\mathfrak{N}_i=\dete\,\Mfrak_i\ \text{and}\ d_i(\bar{\mathfrak{N}}_i,P_i)=R\}\subset\pr_i(C_\nu(A)(\bar\Fbb_p)).\]
As $\pr_i$ is projective, the projection $\pr_i(C_\nu(A))$ is closed. Hence the closure of the above subset is also contained in $\pr_i(C_\nu(A))$, and by maximality of $R$ it has to be all of $\pr_i(C_\nu(A))$. 
\end{proof}
\begin{proof}[Proof of Theorem \ref{theoirred}]
Let $x,y\in C_\nu(A)$. By Proposition $\ref{mainstep}$ there exist $i\in\{1,\dots,n\}$ and $\Mfrak\in C_\nu(A)(\bar\Fbb_p)$ with $\bar\Mfrak_i=Q_i$ such that $\Mfrak_i$ and $x_i$ lie in the same connected component of $\pr_i(C_\nu(A))$. By Lemma $\ref{MQiinC}$ we have $\Mfrak(Q_i)\in C_\nu(A)$. As the fibers of $\pr_i: C_\nu(A)\rightarrow \Fcal_i$ and $\pr_j:C_\nu(A)\rightarrow \Fcal_j$ are connected by Corollary $\ref{Corconnected}$, it follows that $x$ and $\Mfrak(Q_i)$ lie in the same connected component of $C_\nu(A)$.
Similarly there exists $j\in\{1,\dots,n\}$ such that $y$ and $\Mfrak(Q_j)$ lie in the same connected component of $C_\nu(A)$.
Now, using Lemma $\ref{MQiMQj}$ and Corollary $\ref{Corconnected}$ again, it follows that $\Mfrak(Q_i)$ and $\Mfrak(Q_j)$ lie in the same connected component of $C_\nu(A)$ and hence the same is true of $x$ and $y$.
\end{proof}

\section{Application to deformation spaces}
Finally we want to state a consequence of the main result. In this section we assume $p>2$.\\
Let $\bar\rho:G_K\rightarrow \GL_2(\Fbb)$ be an irreducible $2$-dimensional continuous representation of the absolute Galois group $G_K={\rm Gal}(\bar K/K)$ of $K$ with coefficients in a finite extension $\Fbb$ of $\Fbb_p$. We assume that $\bar\rho$ is flat. In this case the flat deformation functor of Ramakrishna is pro-representable by a complete local noetherian $W(\Fbb)$-algebra $R^{\rm fl}$. 

Recall that we considered a subfield $K_\infty$ of $\bar K$ obtained by adjoining a compatible system of $p^n$-th roots of a fixed uniformizer $\pi\in\Ocal_K$. 
Its absolute Galois group $G_{K_\infty}$ is identified with the absolute Galois group of a local field in characteristic $p$ and hence there is a $\phi$-module $(N_{\bar\rho},\Phi_{\bar\rho})$ of rank $2$ over $(k\otimes_{\Fbb_p}\Fbb)((u))$ associated with the restriction to $G_{K_\infty}$ of the Tate-twist $\bar\rho(-1)$.
As $\bar\rho$ is irreducible and flat, the $\phi$-module $(N_{\bar\rho},\Phi_{\bar\rho})$ is simple.

Fix a cocharacter
\[\mu:\Gbb_{m,\bar\Q_p}\longrightarrow (\Res_{K/\Q_p}\GL_2)_{\bar\Q_p}\]
which is dominant with respect to the restriction of the Borel subgroup of upper triangular matrices in $\GL_2$.
The cocharacter is given by a tuple $(\mu_\psi)_\psi$, where $\mu_\psi$ is a dominant cocharacter of $\GL_2$ and $\psi$ runs over all embeddings $K\hookrightarrow \bar\Q_p$.
Assume that $\mu_\psi$ is given by $t\mapsto {\rm diag}(t^{a_\psi},t^{b_\psi})$. We write $R^{{\rm fl},\mu}$ for the quotient of $R^{\rm fl}$ corresponding to those deformations $\xi:G_K\rightarrow \GL_2(\Ocal_E)$, where $\Ocal_E$ is the ring of integers in some extension $E$ of $W(\Fbb)[1/p]$, such that the Hodge-Tate weights of $\xi$ are given by $\mu$, i.e. the jumps of the filtrations on
\[D_{\rm cris}(\xi\otimes_{\Ocal_E}\bar\Q_p)_K\otimes_{K,\psi}\bar\Q_p\]
are given by $\{a_\psi,b_\psi\}$. We may assume that $a_\psi,b_\psi\in\{0,1\}$ as otherwise $R^{{\rm fl},\mu}$ is empty.
\begin{cor}
The scheme $\Spec(R^{{\rm fl},\mu})$ is connected.
\end{cor}
\begin{proof}
We construct a dominant cocharacter
\[\nu:\Gbb_{m,\bar\Fbb_p}\longrightarrow (\Res_{k/\Fbb_p}\GL_2)_{\bar\Fbb_p}\]
as follows: The cocharacter $\nu$ is given by cocharacters $\nu_{\bar\psi}$ of $\GL_2$, where $\bar\psi$ runs over all embeddings $k\hookrightarrow \bar\Fbb_p$.

For each $\bar\psi$ we define $\nu_{\bar\psi}(t)=(t^{\alpha_{\bar\psi}},t^{\beta_{\bar\psi}})$, where (with the notations from above)
\begin{align*}
\alpha_{\bar\psi}&=\sum_{\psi\hspace{-3mm}\mod\pi=\bar\psi}a_\psi \\
\beta_{\bar\psi}&=\sum_{\psi\hspace{-3mm}\mod\pi=\bar\psi}b_\psi.
\end{align*}
By \cite[Corollary 2.4.10]{Kisin} the connected components of $\Spec(R^{{\rm fl},\mu})$ are in bijection with those of $C_\nu(\Phi_{\bar\rho})$ which yields the claim.
\end{proof}

\medskip
\noindent 
{Mathematisches Institut der Universit\"at Bonn\\ Endenicher Allee 60\\ 53115 Bonn, Germany}\\
\emph{E-mail:}\\
{hellmann@math.uni-bonn.de}

\end{document}